\documentclass[letterpaper, reqno, 11pt]{amsart}
\usepackage[english]{babel}
\usepackage{graphicx,color}
\usepackage[all]{xy}
\usepackage{amsfonts,dsfont}
\usepackage{amsmath, amsthm, amssymb}
\usepackage{mathrsfs}
\usepackage{enumerate, url}
\usepackage{fancyhdr}
\usepackage{ifthen,srcltx}
\usepackage{comment}

\newcommand{\bbR}{\ensuremath{\mathbb{R}}}

\newcommand{\bbZ}{\ensuremath{\mathbb{Z}}}
\newcommand{\bbT}{\ensuremath{\mathbb{T}}}

\newcommand{\bbN}{\ensuremath{\mathbb{N}}}

\newcommand{\bbE}{\ensuremath{\mathbb{E}}}
\newcommand{\bbP}{\ensuremath{\mathbb{P}}}
\newcommand{\bbone}{\ensuremath{\mathds{1}}}

\newcommand{\backsslash}{\mathbin{\mkern-2mu \char`\\ \mkern-6mu \char`\\ \mkern-4mu}}

\DeclareMathOperator{\Path}{Path}

\DeclareMathOperator{\supp}{supp}

\numberwithin{equation}{section}

\newtheorem{thm}{Theorem}[section]
\newtheorem{prop}[thm]{Proposition}
\newtheorem{lem}[thm]{Lemma}
\newtheorem{cor}[thm]{Corollary}
\theoremstyle{definition}
\newtheorem{dfn}[thm]{Definition}

\newtheorem{rmk}[thm]{Remark}

\newtheorem{ex}[thm]{Example}

\newtheorem{claim}[thm]{Claim}

\newtheorem{conj}[thm]{Conjecture}

\newtheorem{question}[thm]{Question}
\newtheorem*{assume}{Standing assumption}       

\newcommand{\showcomments}{yes}

\newsavebox{\commentbox}
\newenvironment{mycomment}%
{\ifthenelse{\equal{\showcomments}{yes}}%
{\footnotemark
        \begin{lrbox}{\commentbox}
        \begin{minipage}[t]{1.25in}\raggedright\sffamily\tiny
        \footnotemark[\arabic{footnote}]}
{\begin{lrbox}{\commentbox}}}
{\ifthenelse{\equal{\showcomments}{yes}}
{\end{minipage}\end{lrbox}\marginpar{\usebox{\commentbox}}}
{\end{lrbox}}}

\begin{document}

\title{Heat kernels are not uniform expanders} 

\author[M. Fraczyk]{Mikolaj Fraczyk}
\address{Renyi Institute, Hungarian Academy of Sciences\\ R\'ealtanoda utca 13-15, Budapest 1053\\
Hungary}

\author[W. van Limbeek]{Wouter van Limbeek}
\address{Department of Mathematics, Statistics, and Computer Science \\ 
                 University of Illinois at Chicago \\
                 Chicago, IL 60647 \\
                 USA}
                 
\date{\today}

\begin{abstract}
We study infinite analogues of expander graphs, namely graphs where subgraphs weighted by heat kernels form an expander family. Our main result is that there does not exist any infinite expander in this sense. This proves the analogue for random walks of Benjamini's conjecture that there is no infinite graph whose metric balls are uniformly expander. The proof relies on a study of stationary random graphs, in particular proving non-expansion of heat kernels in that setting. A key result is that any stationary random graph is stationary hyperfinite, which is potentially of independent interest.
\end{abstract}

\maketitle


\section{Introduction}
\label{sec:intro}

A sequence of finite graphs $(\mathcal G_n)_n$ of uniformly bounded degrees is an {\bf expander} sequence if $|\mathcal G_n|\to\infty$ and there exists $\varepsilon>0$ such that every $A\subset \mathcal G_n$ with $|A|\leq |\mathcal G_n|/2$ satisfies $|\partial A|\geq \varepsilon |A|$. Expander graphs are therefore sparse robust networks, which makes them of great utility in applications. One therefore wonders whether instead of a sequence, there is a single infinite graph that is sparse yet robust? 

More precisely, Benjamini defined an infinite, connected, bounded degree graph $\mathcal G$ to be an {\bf expander at all scales} if there exists $\varepsilon>0$ such that every ball $B\subset \mathcal G$ and subset $A\subset \mathcal G$ with $|A\cap B|\leq |B|/2$ satisfy $|\partial A\cap B|\geq \varepsilon |A\cap B|$, and conjectured:
	\begin{conj}[{Benjamini \cite{BenjaminiConj}}] 
	\label{conj:benjamini}
	Expanders at all scales do not exist. 
	\end{conj}
See also \cite{Benjamini1998} for an earlier statement for Cayley graphs and \cite{BenjaminiErdos} for a variant for families of finite graphs.

Explicit expander graphs were first constructed as finite quotients of Cayley graphs of groups $\Gamma$ with Property (T) (or $(\tau)$) by Margulis \cite{margulis_expander}. In this case, the Cayley graph of $\Gamma$ is the Gromov-Hausdorff limit of expander graphs, but nevertheless sometimes (and conjecturally, always) is not an expander at all scales (it can even be a tree). 

It is known that graphs with \textbf{metric property A} (see \cite{NowakAMM}) are never expanders at all scales: Brodzki-Niblo-{\v S}pakula-Willett-Wright proved Property A implies \textbf{uniform local amenability} (ULA for short) \cite{Brodzki} and it is immediate that a ULA graph is not an expander at all scales. Property A holds for Cayley graphs of \textbf{exact groups} (Ozawa \cite{OzawaICM}), which includes e.g. linear groups and hyperbolic groups. Without homogeneity of the underlying graphs, much less is known: Indeed, the closely related notion of a Schreier graph is essentially equivalent to a regular graph.

Our main purpose is to study the analogue of Conjecture \ref{conj:benjamini} that replaces metric balls by distributions of random walks as a measure of robustness. Let $\mathcal G$ be a connected rooted graph of bounded degree with a vertex $x$. We will write $\mu^n_x$ for the distribution of the $n$th step of the simple random walk starting at $x$ (i.e. the heat kernel).
\begin{dfn} Let $h>0$. We say the heat kernel on a rooted graph $(\mathcal G,o)$  is $h$\emph{-expanding} if for every $n$ and $A\subset \mathcal G$ with $\mu^n_o(A)\leq \frac{1}{2}$, we have 
	\begin{equation} 
	\mu^n_o(\partial A)\geq h \, \mu^n_o(A). 
	\label{eq:expansion}
	\end{equation}
We say that the heat kernel on a rooted graph $(\mathcal G,o)$ is \emph{expanding} if there exists $h>0$ such that it is $h$-expanding. Similarly we say that the heat kernels on an (unrooted) graph $\mathcal G$ are expanding if there exists $h>0$ such that for any choice of root $o\in \mathcal G$ the walk on $(\mathcal G,o)$ is $h$-expanding.
\label{dfn:expanding_walks}
\end{dfn}
\begin{rmk} 
\label{rmk:walks_vs_benjamini}
Expansion of the heat kernel is analogous to expanding at all scales: For a bounded degree graph $\mathcal G$, consider the family of measures $\mu_{p,r}$ (for $p\in \mathcal G$ and $r>0$) given by the uniform distribution on the ball $B(p;r)$. Then $\mathcal G$ is an expander at all scales precisely when there exists $h>0$ such that for every $p\in \mathcal G$, $r>0$ and $A\subset \mathcal G$ with $\mu_{p,r}(A)\leq \frac{1}{2}$, we have
	$$\mu_{p,r}(\partial A)\geq h \, \mu_{p,r}(A).$$
\end{rmk}
Our main result is the heat kernel analogue of Conjecture \ref{conj:benjamini}:
 \begin{thm} Let $\mathcal G$ be an infinite, connected, bounded degree graph. Then the heat kernel on $\mathcal G$ is not expanding.
\label{thm:main}
\end{thm}
We do not know whether the analogue of the above theorem holds for rooted graphs, especially since the rooted version of Benjamini's conjecture is false:
\begin{ex}[A rooted expander at all scales] Fix a prime $p$ and an expander family of $p$-congruence quotients $\{G_k\}_k$ of a finitely generated linear group $\Gamma$. Let $\mathcal G$ be the graph with vertex set $\sqcup_k G_k$ and edges given by those in $\{G_k\}_k$ as well as an edge between $x\in G_k$ and its image in $G_{k-1}$. Using that a definite proportion of the points of $\cup_{k\leq n} G_k$ is contained in $G_n$, is not hard to see that the family of metric balls centered at a fixed vertex $o\in\mathcal G$ are expanding.\end{ex}
In the above example, the probability that the $k$th step of the random walk on $(\mathcal G, o)$ lies in $G_n$ decays uniformly in $n$, so the subsets $\cup_{k\leq n} G_k$ have small boundary as measured by the random walk starting at a fixed root, which proves its non-expansion. So we pose the following:
\begin{question} Does there exist a bounded degree rooted graph $(\mathcal G, o)$ with expanding heat kernel? \end{question}
Since we do not know how to establish non-expansion of the heat kernel starting at any given root, a key idea in the proof of Theorem \ref{thm:main} is to sample the root randomly on $\mathcal G$. This leads us to study random walks on random graphs.

\subsection*{Stationary random graphs} \label{sec:stat_intro} First defined and studied by Benjamini-Curien \cite{BenCur}, a \emph{stationary random graph} is a random rooted graph whose distribution is invariant under rerooting at a uniformly random neighbor of the root (see Section \ref{sec:stationary} for details). We establish non-expansion of heat kernels on stationary random graphs:
\begin{thm}\label{thm:mainSt}
Let $(\mathcal G,o)$ be an infinite stationary random graph of bounded degree. Then the heat kernels on $(\mathcal G,o)$ is non-expanding almost surely.
\end{thm}
The proof relies on a number of results on the general structure of stationary random graphs that could be of independent interest: First, we establish a connection with the theory of measured equivalence relations. Indeed, given a measured equivalence relation, any graphing of it yields a random graph and in fact any stationary random graph is obtained in this way (see Section \ref{sec:stat_graphing} for details). This allows us to construct \emph{Poisson boundaries} for stationary random graphs, and we show these are amenable (Proposition \ref{prop:PoissonRSG}). Using the equivalence of amenability and \emph{hyperfiniteness} for non-singular equivalence relations, we deduce that any stationary random graph is hyperfinite (Corollary \ref{cor:SRGareHyperfinite}). Note this is in sharp contrast with the behavior of unimodular random graphs, where hyperfiniteness often fails, e.g. for trees (see Example \ref{ex:hyperfinite_tree}). Hyperfiniteness of stationary random graphs is the key ingredient in the proof of Theorem \ref{thm:mainSt}. In fact, we prove a stronger statement about hyperfiniteness of certain sequences of weighted graphs.

The notion of hyperfiniteness for families $(\mathcal G_i)_i$ of finite graphs was introduced by Elek in \cite{Elek2006}. Informally, this means that each $\mathcal G_i$ can be cut into uniformly bounded pieces with small boundaries by removing an arbitrarily small proportion of vertices. Hyperfiniteness is strictly stronger than non-expansion and in many ways is the correct analogue of amenability in the graph setting. The strengthened version of Conjecture \ref{conj:benjamini} asserts that any bounded degree graph contains a sequence of balls that is hyperfinite.


Hyperfiniteness can be similarly defined for weighted graphs (see \cite{ElekTimar}).  
Our Corollary \ref{cor:SRGareHyperfinite} can then be rephrased as follows: For any stationary random graph $(\mathcal G,o)$, the sequence of graphs $(\mathcal G,o)$ weighted by $\mu_o^n$ is hyperfinite. To deduce Theorem \ref{thm:mainSt}, we use hyperfiniteness to show that with large probability, one has a partition into finite pieces with small boundary. From these partitions we assemble large sets with small boundary that witness the non-expansion of the heat kernels.

\subsection{Outline of the paper} We start by recalling the construction of the Poisson boundary of a group, which enables us to prove Main Theorem \ref{thm:main} in the special case of Cayley graphs (Section \ref{sec:poisson}). Next, we discuss stationary random graphs and graphings (Section \ref{sec:stationary}). In Section \ref{sec:poissonSRG} we construct Poisson boundaries of stationary random graphs, prove they are amenable, and deduce that any stationary random graph is hyperfinite. Finally, in Section \ref{sec:main_proof}, we prove our main theorems, first for stationary random graphs (Theorem \ref{thm:mainSt}) and then for bounded degree graphs (Theorem \ref{thm:main}). 
\subsection*{Acknowledgments} The authors thank Gabor Elek for explaining the connection between the main theorem and Property A. We thank Alex Furman for helpful conversations on amenability of Poisson boundaries. We thank Miklos Abert and Itai Benjamini for valuable comments on an earlier version of this paper. We thank the University of Illinois at Chicago for providing support for a visit by MF. MF is supported by ERC Consolidator Grant 648017. WvL is supported by NSF DMS-1855371 and an AMS-Simons Travel Grant.

\section{Poisson boundaries and random walks on groups} \label{sec:poisson}
\subsection{Poisson boundary}
We start by recalling a few basic facts about Poisson boundaries. Let $G$ be a countable group and let $\mu$ be a probability measure on $G$ such that $\supp(\mu)$ generates $G$. The \textbf{path space} is the product space $G^\bbN$. Let us write $\bbP^\mu$ for the measure on $G^\bbN$ obtained as push-forward of $\mu^\bbN$ via the map 
	$$(x_1,x_2,\ldots)\mapsto (1, x_1,x_1x_2,\ldots).$$
We define the shift operator $T:G^\bbN\to G^\bbN$ as $T((x_i)_{i\in\bbN})=(x_{i+1})_{i\in\bbN}$. The shift operator commutes with the action of $G$ by left-multiplication. The \textbf{Poisson boundary} $(P,\tau)$ is the quotient of the space $(G^\bbN,\bbP^\mu)$ by the $\sigma$-algebra of $T$-invariant sets. Since the support of $\mu$ generates $G$, one can easily verify that the induced action of $G$ on $(P,\tau)$ is non-singular and that the measure $\tau$ is $\mu$-stationary. It is a theorem of Zimmer \cite{Zimmer78} that the action of $G$ on the Poisson boundary is amenable (in the sense of \cite{Zimmer78}). We reproduce the sketch of Zimmer's proof since later we will need to perform a similar argument to prove the amenability of Poisson boundaries of stationary graphings. The key ingredient of Zimmer's argument is the following lemma.
\begin{lem}[{\cite[Theorem 3.3]{Zimmer78}}]\label{lem:QuotientAm} Let $(X,\nu_1)$ be a probability measure space with a non-singular amenable action of the group $G_1\times G_2$ such that $G_2$ is amenable. Let $\mathcal A$ be the $\sigma$-algebra of $G_2$ invariant sets on $X$. Then the quotient of $(X,\nu_1)$ by $\mathcal A$ is a non-singular amenable $G_1$-space.
\end{lem}
Actually, Zimmer uses a version of Lemma \ref{lem:QuotientAm} where $G_2=\bbN$ is generated by a single transformation (see \cite[Section 5]{Zimmer78} for details). In order to use this to prove amenability of the Poisson boundary of $G$, we must first replace $\bbP^\mu$ with a quasi-invariant measure which makes the left $G$-action on $G^\bbN$  amenable. For any probability measure $\nu$ on $G$ the convolution $\nu*\bbP^\mu$ is the push-forward of the measure $\nu\times \mu^\bbN$ via the map
$(g,x_1,x_2,\ldots)\mapsto (g,gx_1,gx_1x_2,\ldots).$ 
We claim that as soon as $\nu$ has full support on $G$ then the action $G\curvearrowright (G^\bbN,\nu*\bbP^\mu)$ is non-singular and amenable. Non-singularity is obvious and to prove amenability, note that the projection onto the first coordinate $(G^\bbN,\nu*\bbP^\mu)\to (G,\nu)$ is a $G$-equivariant measure-preserving map and $G$ acts amenably on $(G,\nu)$. Therefore, by \cite[2.4]{Zimmer78} the action $G\curvearrowright (G^\bbN, \nu*\bbP^\mu)$ is amenable. 

Let $\mathcal A$ be the algebra of $T$-invariant sets on $G^\bbN$. The quotient of $(G^\bbN,\nu*\bbP^\mu)$ by $\mathcal A$ is $(P,\nu*\tau)$. By the variant of Lemma \ref{lem:QuotientAm} for $G_2=\bbN$, the $G$-action on $(P,\nu*\tau)$ is amenable. It remains to produce $\nu$ with full support such that $\nu*\tau=\tau$, for which we can take the sum of convolutions $\nu:=\sum_{n\geq 1} \frac{1}{2^{n+1}} \mu^{\ast n}$. This concludes the proof of the amenability of $G\curvearrowright(P,\tau)$.

We will need one more well-known property of $(P,\tau)$, namely an ergodic theorem for the random walk:
\begin{lem}
For any  $f\in L^\infty(P,\tau)$ and $\tau$-a.e. $x\in P$, we have $\displaystyle\lim_{n\to\infty} \int_G f(gx)\, d\mu^{\ast n}(g)=\int_P f(y) \, d\tau(y).$
\label{lem:ConvConv}
\end{lem} We provide a short proof:
\begin{proof}
Let $\tilde f$ be the pullback of $f$ to $G^\mathbb N$. The lemma will follow once we show that for $\bbP^\mu$-almost every trajectory $(x_i)_{i\in\bbN}$, we have 
$$\lim_{n\to\infty}\int_G \tilde f((gx_i)_{i\in \bbN})d\mu^{\ast n}(g)=\bbE(\tilde f).$$
Because $\tilde f$ is shift-invariant, we can rewrite the left-hand side as 
\begin{align}\label{eq:Shift}
\lim_{n\to\infty}\int_{G^n} \tilde f(1,g_1,g_1g_2,\ldots, g_1\ldots g_n,g_1\ldots g_nx_1,g_1\ldots g_nx_2,\ldots) \\ d\mu(g_1)\ldots d\mu(g_n).\end{align}
By the martingale convergence theorem, we have that for $\bbP^\mu$-almost every trajectory $(y_i)_{i\in\bbN}\in G^\bbN$:
 $$\tilde{f}((y_i)_{i\in\bbN})=\lim_{n\to\infty}\mathbb E(\tilde f((y_i')_{i\in\bbN} \mid y_i'=y_i \textrm{ for } i\leq n)).$$
Since the expected value on the right depends only on $y_1,\ldots, y_n$ we will write $\bbE \tilde f(y_1,\ldots,y_n):=\bbE(\tilde f((y_i')_{i\in\bbN}) \mid y_i'=y_i \textrm{ for } i\leq n))$. Using (\ref{eq:Shift}) we now can deduce that for almost every trajectory $(x_i)_{i\in\bbN}$ we have
\begin{align*}\lim_{n\to\infty}\int_G \tilde f((gx_i)_{i\in \bbN})d\mu^{\ast n}(g)=&\lim_{n\to\infty}\bbE \tilde{f}(1,g_1,\ldots,g_1g_2\ldots g_n)d\mu(g_1)\ldots d\mu(g_n)\\=&\bbE(\tilde f).\end{align*}
\end{proof}
\subsection{Global non-expansion of heat kernels on groups} We are now able to prove the main result in the special case of Cayley graphs of finitely generated groups.
\begin{thm} 
\label{thm:mainGp}
Let $G$ be a finitely generated group and $\mu$ a finitely supported measure on $G$ whose support generates $G$ as a semigroup. Then the heat kernels on $G$ with distribution $\mu$ are not expanding.
\end{thm}
\begin{proof}[Proof of Theorem \ref{thm:mainGp}] Let $(P,\tau)$ be the Poisson boundary of $(G,\mu)$. By the above results of Zimmer \cite{Zimmer78}, the action of $G$ on $(P,\tau)$ is nonsingular, ergodic, and amenable. The idea of the proof is that amenability of the Poisson boundary implies there exist medium-size sets in $P$ that are nearly invariant under the $G$-action. This allows us to define sets in $G$ that are nearly invariant under the random walk. 

The above idea works well as long as the Poisson boundary is nontrivial, so let us first consider the case that  $(P,\tau)$ is trivial. Then $G$ is amenable, so for every  $\varepsilon>0$ there exists a F{\o}lner set $F_\varepsilon$ such that $|\partial F_\varepsilon|\leq \varepsilon|F_\varepsilon|$.  Using the random walk starting at $e\in G$, we have for every $n$:
	$$\sum_{g\in G} \mu_e^n(gF_\varepsilon)= |F_\varepsilon| \qquad \textrm{ and } \qquad  \sum_{g\in G} \mu_e^n(\partial gF_\varepsilon)=|\partial F_\varepsilon|\leq \varepsilon|F_\varepsilon|. $$
It follows that for every $n$ there exists $g_n$ such that $\mu_e^n(\partial g_n F_\varepsilon)\leq \varepsilon \, \mu_e^n(g_n F_\varepsilon)$. In Lemma \ref{lem:ProbDecay} we establish a dispersion property of random walks, immediately implying that for sufficiently large $n$ (depending on $\varepsilon$), we have $\mu_e^n(g_n F_\varepsilon)\leq \frac{1}{2}$. This proves that heat kernels on $G$ are not $\varepsilon$-expanding.

The remaining case is that $(P,\tau)$ is nontrivial. Then a result of Jones-Schmidt \cite{JonesSchmidt} shows that there exists an almost invariant sequence $\{S_n\}_n$ of subsets of $P$, i.e. we have $\tau(\gamma S_n \Delta S_n)\to 0$ for any $\gamma\in\text{supp}(\mu)$. It follows that $\tau(\partial S_n)\to 0$. Further, Jones-Schmidt in fact show we can choose $S_n$ with $\tau(S_n)=\frac{1}{2}$. 

For $p\in P$ and $n\geq 1$, define
	$$A_n(p):= \{\gamma\in G\mid \gamma p \in S_n\}.$$
We claim that for a.e. $p$, these sets $A_n$ are nearly invariant under the random walk. Indeed, by Lemma \ref{lem:ConvConv}, we have for almost every $p\in P$:
	\begin{align*}
	\lim_{k\to\infty} \mu^k(A_n(p))	&=\int_G \bbone_{S_n}(\gamma p) \mu^k(\gamma)\\
							&=\int_P \bbone_{S_n}(p) d\tau(p)\\
							&=\tau(S_n),
	\end{align*}
where we used Lemma \ref{lem:ConvConv} on the second line. Similarly one sees $\mu^k(\partial A_n(p))\to \tau(\partial S_n)$ as $k\to\infty$ for every $n\geq 1$ and a.e. $p$.\end{proof}

\section{Stationary random graphs and graphings}\label{sec:stationary}
\subsection{Stationary random graphs}\label{sec:SRG}
Let $d\in\bbN$ and let $\mathcal M_{\leq d}$ be the moduli space of connected rooted graphs of degree bounded by $d$. We allow multiple edges and loops. The space $\mathcal M_{\leq d}$ is equipped with the following metric. For any rooted graphs $(\mathcal G_1,o_1),(\mathcal G_2,o_2)$ we put 
	$$d((\mathcal G_1,o_1),(\mathcal G_2,o_2)):=2^{-r} \textrm{ where } r=\inf\{n\in\bbN \mid B_{\mathcal G_1}(o_1,n)\not\cong B_{\mathcal G_2}(o_2,n)\}.$$
This metric induces the Gromov-Hausdorff topology on $\mathcal M_{\leq d}$. For a rooted graph $(\mathcal G,o)\in \mathcal M_d$, write $(X_n)_{n\in\bbN}$ for the simple random walk on $\mathcal G$ starting at the root. 
\begin{dfn}[{Benjamini-Curien \cite{BenCur}}] A \emph{stationary random graph} is a random $\mathcal M_d$-valued variable $(\mathcal G,o)$ such that $(\mathcal G,o)$ and $(\mathcal G,X_n)$ have the same distribution.\end{dfn}
For example, a Cayley graph rooted at any vertex is a stationary random graph. We have the following definition of hyperfiniteness for stationary random graphs:
\begin{dfn} A stationary random graph $(\mathcal G, o)$ is \emph{stationary hyperfinite} if for every $\varepsilon>0$, there exists a stationary random subset $S\subset \mathcal G$ with $\bbP(o\in S)\leq \varepsilon$ and such that $\mathcal G\backslash S$ is a union of finite connected components almost surely.\end{dfn}
This should be compared with hyperfiniteness of unimodular random graphs (where $S$ is required to be invariant). In particular, we warn that a unimodular random graph may be stationary hyperfinite yet not hyperfinite as a unimodular random graph:
\begin{ex} For $d>1$, the $2d$-regular rooted tree $(T_{2d},o)$ is a unimodular random graph that is not hyperfinite, but it is stationary hyperfinite because a union of concentric horospheres with a sparse sequence of radii, centered at a random point of the boundary, will partition $T_{2d}$ into finite components. 

More precisely, identify $T_{2d}$ with the Cayley graph of the free group $F_d$ on $d$ generators and let $b\colon\partial T\times F_d\to \bbZ$ be the Busemann cocycle. Choose $\alpha\in \bbR$ irrational and let $\lambda$ the Lebesgue measure on $\bbT:=\bbZ\backslash \bbR$. Consider the system $(\partial T\ltimes \bbT,\nu\times \lambda)$ with the action $(\xi,\theta)\gamma=(\xi\gamma, \theta+\alpha b(\xi,\gamma))$ for $\gamma\in F_d.$ Let $E_\varepsilon=\partial T\times [0,\varepsilon)$ and put $S_\varepsilon=\{\gamma\in F_d \mid x\gamma\in E_\varepsilon\}$ where $x$ is $(\nu\times\mu)$-random. We have $\bbP (o\in S_\varepsilon)=\varepsilon$ and it is easy to verify that $S_\varepsilon$ is always a union of concentric horospheres with bounded gaps so that $T_{2d}\setminus S_{\varepsilon}$ is a union of finite connected components. Since $\nu\times\lambda$ is $F_d$-stationary, $S_\varepsilon$ is  the intersection of $E_\varepsilon$ with the orbit of a random point in a stationary $F_d$-system, so its distribution is stationary under the action of $F_d$. In particular, it is a stationary random subset of $T_{2d}$.
\label{ex:hyperfinite_tree}
\end{ex}


\subsection{Stationary graphings}\label{sec:stat_graphing}
One way to produce random graphs is to realize them as equivalence classes in some non-singular measured equivalence relation $(X,\nu,\mathcal R)$ where edges are given by a finite graphing $(\varphi_i)_{\in I}$ generating $\mathcal R$. Thanks to this  point of view we will be able to attack our problem using the theory of non-singular measured equivalence relations. 

Let $(X,\nu)$ be a probability measure space and let $\varphi_i:X_i\to X$ be a finite family of non-singular measurable maps defined on subsets $U_i$ of $X$. The triple $(X,\nu,(\varphi_i)_{i\in I})$ is called a \textbf{graphing}. We assume that $(\varphi_i)_{i\in I}$ is \textbf{symmetric}, i.e. for each $i\in I$ the map $\varphi_i^{-1}\colon \varphi_i(U_i)\to U_i$ is also in the set $(\varphi_i)_{i\in I}$.  Let $\mathcal R$ be the orbit equivalence relation generated by maps $(\varphi_i)_{i\in I}$. For $x\in X$ define its degree as $\deg(x):=|\{i\in I \mid x\in U_i\}|$. Intuitively $\deg(x)$ is the number of arrows emanating from $x$.

\begin{assume} In this section, all graphings and equivalence relations are assumed to be countable. \end{assume}

A measured graphing yields a random graph in the following way: For every $x\in X$, let $\mathcal G_x$ be the graph with vertex set given by the equivalence class $[x]_{\mathcal R}$ and place  an edge between $y,z\in [x]_{\mathcal R}$ whenever $z=\varphi_i(y)$ for some $i\in I$ (multiple edges are allowed). The graphs $\mathcal  G_x$ have degrees bounded by $|I|$ and are undirected since $(\varphi_i)_{i\in I}$ is symmetric. If we choose a $\mu$-random point $x$, the resulting graph $\mathcal G_x$ is a random rooted graph. 

The properties of $\mathcal G_x$ will depend on the graphing. For example, if the graphing consists of measure preserving maps then the resulting random graph is unimodular (see \cite{ALR2007}). We are mainly interested in stationary graphs. Those will be realized as equivalence relations in \textbf{stationary graphings}.
\begin{dfn}
A finite graphing $(X, \nu, (\varphi_i)_{i\in I})$ is \textbf{stationary} if for every $f\in L^\infty(X,\nu)$ we have 
$$\int_X f(x)d\nu(x)=\int_X \left(\frac{1}{\deg(x)}\sum_{x\in U_i}f(\varphi_i(x))\right)d\nu(x).$$
\end{dfn}
If $(X, \nu, (\varphi_i)_{i\in I})$ is a stationary graphing then $\mathcal G_x$ is a stationary random graph. Conversely, any stationary random graph arises in this way:
\begin{lem}\label{lem:BernoulliGraphing} For every stationary random graph $(\mathcal G,o)$ there exists a stationary graphing $(X,\nu,(\varphi_i)_{i\in I})$ such that $(\mathcal G,o)=(\mathcal G_x,x)$ in distribution.
\end{lem} 
\begin{proof}
The proof is the same as Lovasz's result that unimodular random graphs can be realized by unimodular graphings \cite[18.37]{Lovasz} (see also the proof of \cite[Proposition 14]{AGV14}), but with $\sigma$ being a stationary distribution on the space of rooted graphs.
\end{proof}
The random walk on a finite measured graphing $(X,\nu, (\varphi_i)_{i\in I})$ is the sequence of random variables $(X_n)_{n\geq 0}$ where $X_0$ is a $\nu$-random point of $X$ and for all $n\geq 1$ we have 
	$$\bbP(X_n=y \mid X_{n-1}=x)=\frac{1}{\deg(x)}|\{i\in I \mid \varphi_i(x)=y\}|.$$ 
For a stationary graphing each step $X_n$ will have the same distribution $\nu$.

We recall the definition of a \textbf{hyperfinite} measured equivalence relation. Let $(X,\nu,\mathcal R)$ be a non-singular measured equivalence relation. We say that $\mathcal R$ is hyperfinite if there exists a family of finite measured equivalence relations $(X,\nu, \mathcal E_n)$ with  $\mathcal E_n\subset \mathcal R$ such that $\mathcal R=\bigcup_{n=1}^\infty \mathcal{E}_n.$ By a theorem of Connes-Feldmann-Weiss \cite[Theorem 10]{ConnesFeldmanWeiss} a non-singular equivalence relation is hyperfinite if and only if it is amenable. 

We prove that for any measured equivalence relation, its `tautological bundle' is amenable. Let $(X,\nu, \mathcal R)$ be a non-singular measured equivalence relation. The tautological bundle is the pair $([X],[\nu])$ where $[X]$ is the set $\mathcal R\subset X\times X$ equipped with measure $[\nu]:=\int_X (\delta_x\times c_x) \, d\nu(x)$ where $c_x$ is the counting measure on the equivalence class $[x]_{\mathcal R}$. Points in the space $[X]$ are pairs $(x,y)$ where $x\in X$ and $y\in [x]_\mathcal R.$
\begin{lem}\label{lem:TautER}
Let $\mathcal R'$ be the equivalence relation on $[X]$ generated by $(x_1,y)\sim (x_2,y)$ for $x_1,x_2\in [y]_{\mathcal{R}}$. Then $\mathcal R'$ is non-singular and amenable as an equivalence relation.
\end{lem}
\begin{proof}
We will prove that $\mathcal R'$ is non-singular and hyperfinite, so amenability will follow by \cite{ConnesFeldmanWeiss}. Let $(\varphi_i)_{i\in \bbN}$ be a countable graphing generating $\mathcal R$. For each $x\in X$, equip the graph $\mathcal G_x$ with a path metric $d_{[x]}$, declaring the length of the edge $(x,\varphi_i(x))$ to be $i$. This ensures that the balls in $\mathcal G_x$ are finite.  

For $r>0$, let $\mathcal E_r$ be the equivalence relation on $[X]$ generated by $(x_1,y)\sim (x_2,y)$ for $x_1,x_2\in [y]_{\mathcal R}$ such that $d_{[y]}(y,x_1),d_{[y]}(y,x_2)\leq r.$ The balls in $\mathcal G_y$ are finite for every $y\in X$, so the classes of $\mathcal E_r$ are finite. On the other hand, for every pair $x_1,x_2\in [y]_{\mathcal {R}}$ there exists $r>0$ such that $x_1,x_2$ are in the $r$-ball centered at $y$. We conclude that $\mathcal R'=\cup_{r>0} \mathcal E_r$, so $\mathcal R'$ is indeed hyperfinite.\end{proof}
The following lemma is an easy consequence of hyperfiniteness. 
\begin{lem}\label{lem:FiniteComponents1}
Let $(X,\nu,(\varphi_i)_{i\in I})$ be a finite symmetric graphing generating a non-singular amenable measured equivalence relation $\mathcal R$. Then for every $\varepsilon>0$ there exists $M\in \bbN$ and a subset $Z\subset X$ such that $\nu(Z)\geq 1-\varepsilon$ and the equivalence relation $\mathcal E$ on $Z$ generated by the restrictions of $(\varphi_i)_{i\in I}$ satisfies $|[z]_{\mathcal E}|\leq M$ for $\nu$-almost all $z\in Z$.
\end{lem}
\begin{proof}
By \cite[Theorem 10]{ConnesFeldmanWeiss} there exists an increasing sequence of finite equivalence subrelations $\mathcal E_n\subset \mathcal R$ such that $\mathcal R=\bigcup_{n\in\bbN} \mathcal E_n$. Set
 	$$Z_n:=\{x\in X \mid \varphi_i(x)\in [x]_{\mathcal E_n} \textrm{ for all } i\in I\}.$$ 
Clearly $Z_n\subset Z_{n+1}$ and $\bigcup_{n\in\bbN} Z_n=X$ modulo a null set. Choose $n$ such that $\nu(Z_n)\geq 1-\varepsilon/2$. The function $Z_n\ni z\mapsto |[z]_{\mathcal E_n}|\in \bbN$ is measurable, so there exists $M\in\bbN$ such that $\nu(\{z\in Z_n \mid |[z]_{\mathcal E_n}|\leq M\})\geq 1-\varepsilon$. Set $Z:=\{z\in Z_n \mid |[z]_{\mathcal E_n}|\leq M\}$ and let $\mathcal E$ be the equivalence relation on $Z$ generated by the restrictions of $(\varphi_i)_{i\in I}$. Clearly $[z]_{\mathcal E}\subset [z]_{\mathcal E_n}$, so every class of $\mathcal E$ has at most $M$ elements.
\end{proof}
Finally, we discuss hyperfiniteness for stationary random graphings.
\begin{dfn} A symmetric stationary graphing $(X,\nu,(\varphi_i)_{i\in I})$ is \emph{hyperfinite} if the measured equivalence relation generated by $(\varphi_i)_{i\in I}$ is hyperfinite.
\end{dfn}
We have the following relationship between hyperfiniteness of a stationary random graph and graphings:
\begin{lem} A stationary random graph $(\mathcal G,o)$ is hyperfinite if and only if there exists a hyperfinite symmetric stationary graphing $(X,\nu,(\varphi)_{i\in I})$ such that $(\mathcal G,o)=(\mathcal G_x,x)$ in distribution. 
\label{lem:equiv_hyperfinites}
\end{lem}
\begin{proof} Suppose first $(X,\nu,(\varphi_i)_{i\in I})$ is hyperfinite. Let $\varepsilon>0$, choose $Z\subset X$ as in Lemma \ref{lem:FiniteComponents1} and set $S:=X\backslash Z$. In particular, $\nu(S)\leq \varepsilon$ and $\mathcal{R}$ has finite equivalence classes on the complement of $P$, so $(\mathcal G_x, x)$ is hyperfinite.

Conversely, suppose $(\mathcal G,o)$ is hyperfinite. and choose stationary random subsets $S_n$ such that $\mathbb{P}(o\in S_n)\leq \frac{1}{n}$ and $\mathcal G \backslash S_n$ is a union of finite components almost surely. For technical reasons assume that $S_n$ have no symmetries a.s. meaning that $(\mathcal G,o,S_n)\neq (\mathcal G,o',S_n)$ for any two choices of the root $o$ and $o'$. This can be always arranged by adding to $S_n$ a small intensity Bernouill percolation on $\mathcal G$. 

We will construct a stationary hyperfinite graphing realizing the stationary random graph $(\mathcal G, o)$ on a suitable moduli space of decorated graphs: Let $Y^1,Y^\infty$ be the space of connected, rooted graphs of degree at most $d$ decorated with a subset or a sequence of subsets respectively. We represent elements of $Y^1$ as $(\mathcal H,o,A)$ where $\mathcal H$ is a connected (deterministic) graph of degree at most $d$ with root $o$ and $A\subset \mathcal H$. Similarly we represent elements of $Y^\infty$ as $(\mathcal H,o,(A_i)_{\in\bbN})$ where $A_i\subset \mathcal H$ for $i\in\bbN$. The standard graphing on $\mathcal M_d$ (re-rooting to a neighbor) lifts to unique symmetric graphings on $Y^1$ and $Y^\infty$, which we will also call the standard graphing. Write $\mathcal R^\infty$ for the re-rooting relation on $Y^\infty$. For $i\in\bbN$, let $\pi_i\colon Y^\infty\to Y^1$ be the projection $(\mathcal H,o,(A_j)_{j\in\bbN})\mapsto (\mathcal H,o,A_i)$, and note that $\pi_i$ is equivariant with respect to the standard graphings. Write $\mu_i$ for the distribution of $(\mathcal G,o, S_i)$ in $Y^1$. Since $S_i$ is a stationary random subset, $\mu_i$ is a stationary probability measure on $Y^1$. Moreover, since $S_i$ have no symmetries a.s. we have $(\mathcal G,o)=(\mathcal G_x,x)$ in distribution for $\mu_i$-random $x\in Y^1$. Finally, let $\mu$ be any stationary coupling of $(\mu_i)_{i\in\bbN}$ on $Y^\infty$, i.e. a stationary probability measure on $Y^\infty$ such that $(\pi_i)_*(\mu)=\mu_i$ for all $i$. Any weak-* limit of ergodic averages of the random walk starting at a $(\prod_i \mu_i)$-random point is such a coupling almost surely. 

We claim that $(Y^\infty,\mu)$ with the re-rooting equivalence relation is hyperfinite. Indeed, let $Z_n:=\{(\mathcal H,o,(A_i)_{i\in \bbN}) \mid o\not\in A_n\}\subset Y^\infty$. If $x=(\mathcal H,o,(A_i)_{i\in\bbN})$ then $(\mathcal{G}_x,x,[x]\cap Z_n)=(\mathcal H,o,\mathcal H\setminus A_n)$. But the distribution of the triple $(\mathcal H,o,A_n)$ is given by $\mu_i$ so for $\mu$-random $x$ we have $(\mathcal{G}_x,x,[x]\cap Z_n)=(\mathcal G,o,\mathcal G\setminus S_n)$ in distribution. Since $(\mathcal G\setminus S_n)$ is a union of finite connected components a.s., we deduce that the relation $\mathcal E_n$ on $Y^\infty$ generated by the standard graphing restricted to $Z_n$ is finite a.s. Finally, we have $\mu(Z_n)=\bbP (o\not\in S_n)\geq 1-\frac{1}{n}$, so $\mathcal R^\infty=\bigcup \mathcal E_n$ modulo a null set. We deduce that $(Y^\infty, \mu,\mathcal R^\infty)$ with the standard graphing is stationary hyperfinite. 
 \end{proof}

\section{Poisson boundaries of stationary graphings}\label{sec:poissonSRG} 
The goal of this section is to develop a good (in particular, amenable) notion of Poisson boundary of a stationary graphing $(X,\nu,(\varphi_i)_{i\in I})$. As in the classical case of groups, it is the quotient of a path space by the shift operator, but in our case the paths traverse different graphs depending on the initial point, so that the path space is a bundle over $X$. Amenability of this Poisson boundary will then follow from amenability of the tautological bundle and the following lemma that the quotient of an amenable equivalence relation by a single transformation is amenable, which is the analogue of Zimmer's \cite[Theorem 3.3]{Zimmer78} (see Lemma \ref{lem:QuotientAm}).
\begin{lem}\label{lem:QuotientAMforER} Let $(X,\nu)$ be a standard Borel probability measure space with a non-singular amenable equivalence relation $\mathcal R$ and suppose $T:X\to X$ is a non-singular measurable map that preserves a graphing of $\mathcal R$. Then the quotient relation $\overline{\mathcal{R}}$ on $T\backsslash X$ is amenable.
\end{lem}
Since the proof is essentially identical to the proof of Lemma \ref{lem:QuotientAm}, we omit it here. One only needs to substitute equivalence relations for group actions in all relevant definitions (such as cocycles), see e.g. \cite[3.3]{MR1242044} for details. The following proposition constructs Poisson boundaries for stationary graphings:
\begin{prop}\label{prop:PoissonRSG}
Let $(X,\nu,\mathcal R)$ be a non-singular measured equivalence relation. Let $(\varphi_i)_{i\in I}$ be a finite symmetric stationary graphing generating $\mathcal R$. Then there exists a space $(P(X),\widetilde{\nu})$ with a stationary graphing  $(\widetilde{\varphi_i})_{i\in I}$ and a measurable map $\pi\colon (P(X),\widetilde{\nu})\to (X,\nu)$ such that
\begin{enumerate}
\item $\pi_*(\widetilde{\nu})=\nu$,
\item $\pi\circ \widetilde{\varphi_i}=\varphi_i\circ \pi$ for all $i\in I$,
\item The equivalence relation $\widetilde{\mathcal R}$ generated by $(\widetilde{\varphi_i})_{i\in I}$ is amenable. 
\end{enumerate}
\end{prop}
\begin{proof}
For a bounded degree graph $\mathcal G$, define its path space $$\Path(\mathcal G):=\{(v_n)_{n\in\bbN}\subset \mathcal G \mid v_{n+1}\sim v_n\}.$$ We endow the space $\Path(\mathcal G)$ with the topology induced from $\mathcal G^\bbN$. The shift operator $T$ acts on $\Path(\mathcal G)$ by $T((v_n)_{n\in\bbN}):=(v_{n+1})_{n\in\bbN}.$ Write $C_{a_0,a_1,\ldots, a_m}$ for the cylinder $\{(v_n)_{n\in\bbN} \mid v_i=a_i \textrm{ for } i=0,\ldots, m\}.$ 
Equip $\Path(\mathcal G)$ with the measure $\bbP_o$ defined by $\bbP_o(C_{a_0,\ldots,a_m}):=0$ if $a_0\neq o$ and $\bbP_o(C_{a_0,\ldots,a_m}):=\prod_{i=0}^{m-1} \deg(a_i)^{-1}$ otherwise. This is the probability measure associated to the simple random walk on $\mathcal G$ starting at the root $o$. For $k\geq 1$ we define $\bbP_o^k:=T^k_*\bbP_o$. 

We can now define the total path space. Let $\Path(\mathcal M_{\leq d})$ be the moduli space of pairs consisting of a rooted graph $(\mathcal G,o)$ and a path $(v_i)_{i\in \bbN}\in \Path(\mathcal G)$, equipped with the topology of uniform convergence on finite balls around the root. Let $F:\Path(\mathcal M_{\leq d})\to\mathcal M_{\leq d}$ be projection onto $\mathcal M_{\leq d}$. We define $\Path(X)$ as the fibered product over $\mathcal M_{\leq d}$ with respect to $F$ and the map $X\ni x\mapsto (\mathcal G_x,x)\in \mathcal M_{\leq d}$, i.e.
$$\Path(X)=\Path(\mathcal M_{\leq d}) \underset{\mathcal M_{\leq d}}{\times} X:=\{(\mathcal G,o,(v_i)_{i\in \bbN})\times x \mid (\mathcal G_x,x)\simeq (\mathcal G,o)\}.$$
The measured structure is induced from the product. The probability measure $\bbP^k$ on $\Path(X)$ is given by the integral: 
	$$(\Path(X),\bbP^k):=\int_X (\Path(\mathcal G_x),\bbP_x^k) \, d\nu(x).$$
We think of a point in the space $\Path(X)$ as a pair $(x, (v_n)_{n\in\bbN})$ where $x\in X$ and $(v_n)_{n\in\bbN}$ is a trajectory in $\Path(\mathcal G_x).$ The natural projection map $\pi_0: \Path(X) \to X$ given by $\pi_0(x,(v_n)_{n\in\bbN}):=x$ satisfies $(\pi_0)_*\bbP^0=\nu$. The shift operator $T$ on $\Path(X)$ is defined fiberwise: $T(x,(v_n)_{n\in\bbN}):=(x,(v_{n+1})_{n\in\bbN})$, so in particular $\pi_0 \circ T=\pi_0.$ 

The lifts $\widetilde\varphi_i$ are defined as follows. For $x\in X_i$ and $(v_n)_{n\in\bbN}$ we put $\widetilde\varphi_i(x,(v_n)_{n\in\bbN})=(\varphi_i(x),(v_n)_{n\in\bbN}).$ This is well-defined since $\varphi_i(x)$ is in the same equivalence class as $x$, so $\mathcal{G}_{\varphi_i(x)}$ and $\mathcal{G}_{x}$ are the same graphs and $\Path(\mathcal{G}_x)=\Path(\mathcal{G}_{\varphi_i(x)}).$ We note that $\widetilde\varphi_i$ commute with $T$ and $\pi_0\circ \widetilde{\varphi_i}=\varphi_i\circ \pi_0$. 

Let $\mathcal A$ be the $\sigma$-algebra of $T$-invariant Borel subsets of $\Path(X)$. Define the Poisson boundary $(P(X),\widetilde{\nu})$ as the Mackey point realization of the quotient of $(\Path(X),\bbP^0)$ by $\mathcal A$. Since $T$ preserves the fibers, the projection map $\pi_0\colon\Path(X)\to X$ factors through $P(X)$. This gives a map $\pi \colon P(X)\to X$ such that $\pi_*\widetilde{\nu}=\nu$. The maps $\widetilde\varphi_i$ commute with $T$ so they descend to maps $\widetilde\varphi_i$ on $P(X)$. We have $\pi\circ \widetilde{\varphi_i}=\varphi_i\circ \pi$. This proves (1) and (2).

To prove (3), let $\mathcal R_0,\mathcal R_1$ be the equivalence relations on $\Path(X)$ and $P(X)$ respectively that are generated by $(\widetilde \varphi_i)_{i\in I}$.  Set $\bbP':=\sum_{k=0}^\infty 2^{-k-1}{\bbP^k}.$ We will show that $\mathcal R_0$ is a non-singular amenable equivalence relation on $(\Path(X),\bbP')$.  Let $([X],[\nu],\mathcal R')$ be the `tautological bundle' over $(X,\nu)$ as defined immediately prior to Lemma \ref{lem:TautER}. The map 
	\begin{align*}
	\Path(X) &\longrightarrow [X]	\\
	 (x,(v_n)_{n\in\bbN})&\mapsto (x,v_0)\
	\end{align*} 
that forgets the trajectory except the initial point, is non-singular\footnote{The random walk explores the entire graph so any starting point $v_0$ is achieved with positive probability.} and maps equivalence classes of $\mathcal{R}_0$ to those of $\mathcal{R}'$. Hence, $([X],[\nu],\mathcal R')$ is a factor of $(\Path(X),\bbP,\mathcal R_0)$. Since  it is amenable by Lemma \ref{lem:TautER}, $(\Path(X),\bbP',\mathcal R_0)$ is an amenable equivalence relation by \cite[Theorem 2.4]{Zimmer78}.

The quotient of $(\Path(X),\bbP',\mathcal R_0)$ by $\mathcal A$ is still the space $(P(X),\widetilde{\nu},\mathcal R_1)$. By Lemma \ref{lem:QuotientAMforER},  amenability passes to quotients by a single transformation, so we deduce that $\mathcal R_1$ is an amenable equivalence relation. \end{proof}
\begin{cor}\label{cor:SRGareHyperfinite} Every stationary random graph of bounded degree is a hyperfinite stationary random graph.\end{cor}
\begin{proof}
Let $(\mathcal G,o)$ be a stationary random graph of degree at most $d$. Let $(X,\nu,(\varphi_i)_{i\in I})$ be a stationary graphing with $|I|\leq 2d$ such that $(\mathcal G,o)=(\mathcal G_x,x)$ in distribution. Let $(P(X),\widetilde{\nu},(\widetilde\varphi_i)_{i\in I})$ be the Poisson boundary constructed in Proposition \ref{prop:PoissonRSG} and write $\widetilde{\mathcal R}$ for the equivalence relation generated by $(\widetilde\varphi_i)_{i\in I}$. Finally let $(\mathcal H_y,y)$ be the stationary random graph associated to the stationary graphing $(P(X),\widetilde{\nu},(\widetilde{\varphi}_i)_{i\in I})$. The equivalence relation $\widetilde{\mathcal R}$ is amenable by Proposition \ref{prop:PoissonRSG}, and hence hyperfinite by \cite[Theorem 10]{ConnesFeldmanWeiss}. We deduce that $(\mathcal H_y,y)$ is a hyperfinite stationary random graph. 

It remains to prove that $(\mathcal H_y,y)=(\mathcal G_x,x)$ in distribution. The map $\pi$ from Proposition \ref{prop:PoissonRSG} induces a graph cover $\pi\colon (\mathcal H_y, y)\to (\mathcal G_{\pi(y)},\pi(y)).$  The definition of maps $\widetilde\varphi_i$ on the path space $\Path(X)$ (before taking the quotient by the $\sigma$-algebra $\mathcal A$) immediately implies that $\pi\colon (\mathcal H_y, y)\to (\mathcal G_{\pi(y)},\pi(y))$ is a graph isomorphism. Since $\pi_*(\widetilde{\nu})=\nu$ we infer that  $(\mathcal H_y,y)=(\mathcal G_x,x)$ in distribution. \end{proof}
We remark that Corollary \ref{cor:SRGareHyperfinite} does not contradict the fact that there are non-hyperfinite unimodular random graphs. Even if the stationary random graph $(\mathcal G,o)$ is unimodular, the graphing that shows stationary hyperfiniteness is not necessarily measure-preserving.

\section{Proof of the Main Theorem}\label{sec:main_proof}
\subsection{Proof for stationary random graphs}\label{sec:stat_proof}
\begin{proof}[Proof of Theorem \ref{thm:mainSt}]
Let $(\mathcal G,o)$ be an infinite connected stationary random graph of degree at most $d$ and let $\varepsilon>0$. For technical reasons we require $\sqrt{(d+1)\varepsilon}\leq 1/10.$  By Corollary \ref{cor:SRGareHyperfinite} there exists a stationary graphing $(X,\nu,(\varphi_i)_{i\in I})$ such that the relation generated by $ (\varphi_i)_{i\in I}$ is hyperfinite and $(\mathcal G,o)=(\mathcal G_x,x)$ in distribution.  By Lemma \ref{lem:FiniteComponents1} there exists a subset $Z$ of $X$ and a constant $M>0$ such that $\nu(Z)\geq 1-\varepsilon$ and the classes of the equivalence relation on $Z$ generated by $(\varphi_i)_{i\in I}$ are of size at most $M$. For $x\in X$ set $\mathcal{F}_x:=[x]_{\mathcal R}\cap Z$ and $\mathcal E_x:=\mathcal G_x\setminus \mathcal F_x$. Then $\mathcal F_x$ is a subgraph of $\mathcal G_x$ such that $\bbP(x\in \mathcal F_x)=\nu(Z)\geq 1-\varepsilon$ and the connected components of $\mathcal F_x$ have at most $M$ vertices. Recall that $\mu_x^n$ is the distribution of the $n$th step of a simple random walk on $(\mathcal G_x,x)$. By stationarity we have
	$$\int \mu_x^n(\partial \mathcal F_x\cup \mathcal E_x) \, d\nu(x)=\int \bbP(X_n\in \partial \mathcal F_x) \, d\nu(x) +\bbP(x\in \mathcal E_x),$$ 
where as before $X_n$ is the $n$th step of the random walk associated to $\mu$. We estimate the first term on the right-hand side as follows:
	\begin{align*}
	\bbP(X_n\in \partial \mathcal F_x)	&=		\frac{\bbP(X_n\in \partial \mathcal F_x \text{ and } X_{n+1}\in \mathcal E_x)}{\bbP(X_{n+1}\in \mathcal E_x \mid X_n \in \partial \mathcal F_x)}\\
						&\leq 	\frac{\bbP(X_{n+1}\in \mathcal E_x)}{d^{-1}}\\
						&=		d \, \bbP(x\in \mathcal E_x),
	\end{align*}
where on the second line, we estimated the denominator using that any point in $\partial \mathcal F_x$ has an edge with endpoint in $\mathcal E_x$ and that the degree is bounded by $d$, and on the final line we used stationarity again. Combining these estimates and using $\bbP(x\in \mathcal E_x)\leq \varepsilon$, we have
	$$\int \mu_x^n(\partial \mathcal F_x\cup \mathcal E_x) \, d\nu(x) \leq (d+1)\varepsilon,$$
so that by Fatou's lemma 
$$\int \liminf_{n\to\infty} \mu_x^n(\partial \mathcal F_x) \,  d\nu(x)\leq (d+1)\varepsilon.$$ 
It follows that the set $\displaystyle X_\varepsilon:=\{x\in X \mid  \liminf_{n\to\infty} \mu_x^n(\partial \mathcal F_x)\leq \sqrt{(d+1)\varepsilon}\}$ has large mass:
\begin{equation}\label{eqn:MeasureOfGoodSet} 
\nu(X_\varepsilon)\geq 1-\sqrt{(d+1)\varepsilon}. 
\end{equation}

\begin{claim}\label{claim:expansion} For every $x\in X_\varepsilon$ the heat kernels on $(\mathcal G_x,x)$ are not $(6\sqrt{(d+1)\varepsilon})$-expanding. \end{claim}
\begin{proof} Let $x\in X_\varepsilon$. In the below Lemma \ref{lem:ProbDecay}, we establish a uniform flattening property for the random walk on $\mathcal{G}_x$, which implies there exists $n\in\bbN$ such that for every $v\in \mathcal G_x$ we have 
\begin{equation}\label{eq:PtBound} 
\mu_x^n(\{v\})\leq \frac{\sqrt{(d+1)\varepsilon}}{3M}.
\end{equation} 
By increasing $n$ if necessary, we can also ensure that $\mu_x^n(\partial \mathcal F_x\cup \mathcal E_x)\leq 2{\sqrt{(d+1)\varepsilon}}$. The last inequality implies that $\mu_x^n(\mathcal E_x)\leq 2{\sqrt{(d+1)\varepsilon}}$ so $\mu_x^n(\mathcal F_x)\geq 1- 2{\sqrt{(d+1)\varepsilon}}>2/3$. Let us enumerate the connected components of $\mathcal F_x$ as $C_1,C_2,\ldots.$ Let $k$ be the smallest integer such that $\sum_{i=1}^k \mu_x^n(C_i)\geq 1/2$ and define $S_x:=\bigcup_{i=1}^{k-1} C_i$. By (\ref{eq:PtBound}) we have 
$$\frac{1}{2}\geq \mu_x^n(S_x)\geq \frac{1}{2}-\mu_x^n(C_k)\geq \frac{1}{2} -M\frac{\sqrt{(d+1)\varepsilon}}{3M}> \frac{1}{3}.$$
On the other hand $\partial S_x\subset \partial F_x$ so $\mu^n_x(\partial S_x)\leq \mu^n_x(\partial \mathcal F_x)\leq 2\sqrt{(d+1)\varepsilon}.$ Hence, 
$$\frac{\mu_x^n(\partial S_x)}{\mu_x^n(S_x)}<6\sqrt{(d+1)\varepsilon}.$$ This proves the claim. \end{proof}

To prove the theorem, set $X_0:=\liminf_{m\to\infty} X_{m^{-4}}$. By \eqref{eqn:MeasureOfGoodSet}, the Borel-Cantelli lemma applies to the complement of $X_0$ and shows that $\nu(X_0)=1$. Further by Claim \ref{claim:expansion} we know that for every $x\in X_0$ the heat kernels on $(\mathcal G_x,x)$ are not expanding. This completes the proof since $(\mathcal G,o)=(\mathcal G_x,x)$ in distribution. 
\end{proof}
In the above proof of Theorem \ref{thm:mainSt}, we needed uniform flattening of random walks on infinite graphs. This is established in the following lemma.
\begin{lem}\label{lem:ProbDecay}
Let $(\mathcal G,o)$ be a infinite bounded degree connected rooted graph. Let $(X_n)_{n\in\bbN}$ be the simple random walk on $\mathcal G$ starting at $o$. Then 
	$$\displaystyle\lim_{n\to\infty}\max_{v\in \mathcal G}\bbP(X_n=v)=0.$$ 
\end{lem}
\begin{proof}
Let $d$ be the maximal degree of $\mathcal G$. Let $c_n=\max_{v\in \mathcal G}\frac{\bbP(X_n=v)}{\deg(v)}$ and write $c:=\limsup_{n\to\infty}c_n$. We need to show that $c=0$. Suppose to the contrary that $c>0$. 
Our first observation is that for any $n\in\bbN$ we have 
\begin{equation}\label{eq:Step}\frac{\bbP(X_{n+1}=v)}{\deg(v)}=\frac{1}{\deg(v)}\sum_{w\sim v}\frac{\bbP(X_{n}=w)}{\deg(w)},
\end{equation} so $c_{n+1}\leq c_n$. In particular we have $c_n\geq c$ for every $n\in \bbN$. Choose $m\in\bbN$ such that $(m-1)c > 1$ and  $n> m$ such that $c_n\leq (1+d^{-2m})c.$ Let $v_0\in \mathcal G$ be such that $\bbP(X_{n+2m}=v_0)=c_{n+2m}$ and choose vertices $v_1,\ldots, v_m$ such that $d(v_0,v_i)=2i$ for $i=1,\ldots, m$. Let $W_{2m}$ be the set of walks of length $2m$ starting at $v_0$. We have $|W_{2m}|\leq d^{2m}$ and for each $i$ there is at least one walk that ends in $v_i$. For each walk $\underline{w}=(w_0,w_1,\ldots,w_{2m})$ we write $\deg(\underline w):=\prod_{i=0}^{2m-1}\deg(w_i).$ Applying formula \eqref{eq:Step} $2m$ times, we get 
\begin{align*}c\leq \frac{\bbP(X_{n+2m}=v_0)}{\deg(v_0)}=&\sum_{\underline w\in W_{2m}} \frac{\bbP(X_n=w_{2m})}{\deg(\underline w)\deg(w_{2m})}
\end{align*}
Choose walks $\underline{w}^i, i=1,\ldots,m$ such that $w_{2m}^i=v_i$. Considering these separately, we compute
	$$c\leq \sum_{\underline w\in W_{2m}\setminus\{\underline w^1,\ldots, \underline w^m\}} \frac{c_n}{\deg(\underline w)}+\sum_{i=1}^m \frac{\bbP(X_n=v_i)}{\deg(\underline{w}^i)\deg(v_i)}$$
A simple inductive argument shows that $\sum_{\underline w\in W_{2m}}\frac{1}{\deg(\underline w)}=1$. Hence 
	$$c\leq c_n\left( 1-\sum_{i=1}^m \frac{1}{\deg(\underline w^i)}\right)+\sum_{i=1}^m \frac{\bbP(X_n=v_i)}{\deg(\underline{w}^i)\deg(v_i)}.$$
Rearranging and using $c_n\leq (1+d^{-2m})c$, we have
	$$\sum_{i=1}^m\frac{1}{\deg(\underline w^i)}\left(c_n- \frac{\bbP(X_n=v_i)}{\deg(v_i)}\right)\leq c_n-c\leq d^{-2m}c.$$
Since the terms in the sum on the left-hand side are nonnegative and $\deg(\underline w^i)\leq d^{2m}$ for all $i$, multiplying both sides by $d^{2m}$ shows
	$$\sum_{i=1}^m\left(c_n- \frac{\bbP(X_n=v_i)}{\deg(v_i)}\right)\leq c.$$
Since $m$ was chosen such that $mc-c>1$, we obtain
	$$\sum_{i=1}^m \frac{\bbP(X_n=v_i)}{\deg v_i} \geq mc_n-c\geq mc-c>1.$$
This contradicts $\sum_{v\in \mathcal G} \bbP(X_n=v)=1$.
\end{proof}

\subsection{Proof for bounded degree graphs} \label{sec:graph_proof}

\begin{proof}[Proof of Theorem \ref{thm:main}] We argue by contradiction, so suppose $\mathcal{G}$ is a connected infinite graph of degree at most $d$ such that the heat kernels of $\mathcal G$ rooted at starting at any vertex are $\varepsilon$-expanding. 
\begin{lem}\label{lem:GHlimit} Let $(o_n)_{n\in\bbN}$ be a sequence of vertices of $\mathcal G$ such that $(\mathcal G,o_n)$ converges to $(\mathcal G',o)$ in Gromov-Hausdorff topology. Then the heat kernels on $(\mathcal G',o)$ are $\varepsilon$-expanding. 
\end{lem} 
\begin{proof}
Let $m\in\bbN$ and let $S$ be a subset of vertices of $\mathcal G'$. We need to prove that either $\mu_o^m(S)\geq \frac{1}{2}$ or $\mu_o^m(\partial S)\geq \varepsilon \mu_o^m(S).$ Since $(\mathcal G,o_n)$ converge to $(\mathcal G',o)$, for sufficiently large $n$ the rooted graphs $B_{\mathcal G}(o_n,m)$ and $B_{\mathcal G'}(o,m)$ are isomorphic. Choose $n_0$ such that this holds and fix a root-preserving isomorphism $\iota: B_{\mathcal G}(o_{n_0},m) \to B_{\mathcal G'}(o,m)$, and put $S_0:=\iota^{-1}(S\cap B_{\mathcal G'}(o,m)).$ Then $\mu_o^m(S)=\mu_{o_{n_0}}^m(S_0)$ and $\mu_o^m(\partial S)=\mu_{o_{n_0}}^m(\partial S_0)$ because the distribution of first $m$ steps of a random walk depends only on the $m$-neighborhood of the root. Because the heat kernels on $\mathcal G$ are $\varepsilon$-expanding, we will have either $\mu_o^m(S)\geq \frac{1}{2}$ or $\mu_o^m(\partial S)\geq \nu_o^m(S).$
\end{proof}

We will now construct a stationary random graph $(\mathcal G'',o)$ which is almost surely $\varepsilon$-expanding. Fix any initial root $o\in \mathcal G$ and consider the sequence of probability measures $\{\nu_N\}_N$ on $\mathcal{M}_{\leq d}$ supported on graphs isomorphic to $\mathcal G$ with root distributed according to the first $N$ steps of the random walk on $\mathcal G$ starting at $o$, i.e.
  	$$\nu_N:=\frac{1}{N}\sum_{k=0}^{N-1} \int_{\mathcal G}\delta_{\mathcal G,v} \, d\mu_{o}^{k}(v).$$
If we choose a $\nu_N$-random graph and move the root to a random neighbor, the distribution of the resulting graph is the same average with $k$ replaced by $k+1$. It follows that any weak-* limit $\nu$ of $\nu_N$ is the distribution of a stationary random graph supported on Gromov-Hausdorff limits of rooted graphs isomorphic to $\mathcal G$. By Lemma \ref{lem:GHlimit}, the limit $\nu$ is almost surely $\varepsilon$-expanding, but this contradicts Theorem \ref{thm:mainSt}.
\end{proof}

\bibliographystyle{alpha}
\bibliography{ref}

\end{document}